\documentclass[a4paper,DIV=classic]{myart}
\KOMAoptions{DIV=last}

\newcommand{\abs}[1]{\left\vert#1\right\vert}
\newcommand{\Set}[1]{\ensuremath{ \left\{ #1 \right\} }}
\newcommand{\set}[1]{\ensuremath{ \{ #1 \} }}
\newcommand{\R}{\mathbb{R}}
\newcommand{\N}{\mathbb{N}}
\renewcommand{\mid}{\,|\,}

\newcommand*{\cadlag}{c\`adl\`ag}
\newcommand*{\caglad}{c\`agl\`ad}
\newcommand*{\ladlag}{l\`adl\`ag}

\begin{document}

\title{Minimal supersolutions of BSDEs under volatility uncertainty}

\author[a,1,s]{Samuel Drapeau}
\author[a,2,t]{Gregor Heyne}
\author[b,3,s]{Michael Kupper}

\address[a]{Humboldt-Universit\"at zu Berlin, Unter den Linden 6, D-10099 Berlin}
\address[b]{Universit\"at Konstanz, Universit\"atstra\ss e 10, 78464 Konstanz}
\eMail[1]{drapeau@math.hu-berlin.de}
\eMail[2]{heyne@math.hu-berlin.de}
\eMail[3]{kupper@uni-konstanz.de}


\myThanks[s]{Funding: MATHEON project E.11}
\myThanks[t]{Funding: MATHEON project E.2}
\date{\today}

\abstract{
	We study the existence of minimal supersolutions of BSDEs under a family of mutually singular probability measures.
	We consider generators that are jointly lower semicontinuous, positive, and either convex in the control variable and monotone in the value variable, or that fulfill a specific normalization property.
}
\keyWords{Minimal Supersolutions of Second Order Backward Stochastic Differential Equations; Model Uncertainty; $G$-Expectation}

\maketitle

\section{Introduction}

We study the existence of minimal supersolutions of BSDEs under a general family of mutually singular probability measures.
To that end we consider a probability space $(\Omega,\mathcal{F},P)$ carrying a Brownian motion $W$.
By $(\mathcal{F}_t)$ we denote the Brownian filtration.
Given a family $\Theta$ of volatility processes $\theta$, we consider the process $\tilde{W}:\tilde{\Omega}\times \left[ 0,T \right]\to \mathbb{S}^{>0}_d$ defined as the stochastic integral
\begin{equation*}
	\tilde{W}\left( \theta \right)=\int\theta^{1/2} dW,\quad \theta \in \Theta,
\end{equation*}
where $\tilde{\Omega}:=\Omega \times \Theta$.
It generates a raw filtration $\tilde{\mathcal{F}}_t:=\sigma(\tilde{W}_s ; s\leq t)$, $t \in [0,T]$.
The family of measures is now given by $P^\theta[A]:=P[A(\theta)]$, $\theta\in\Theta$, for $A\in \tilde{\mathcal{F}}_T$ and in general it is not possible to define a probability measure under which all probability measures $P^\theta$ are absolutely continuous.

Following the approach developed in \citet{DHK1101} and \citet{HKM1011} we aim at constructing the candidate value process for the minimal supersolution of a BSDE by taking the essential infimum at each point in time and obtaining the corresponding control process by some compactness arguments.
Since the definition of an essential infimum over a set of random variables depends strongly on the underlying probability measure we first provide conditions under which it is possible to define a related notion.
More precisely, this is done by only minimizing over random variables with a specific regularity structure.
Moreover, by assuming that the set of probability measures is relatively compact we also obtain the existence of a sequence approximating the infimum in the capacity sense.

With this at hand, the next step is to adjust the framework of \citep{DHK1101} and \citep{HKM1011} in order to incorporate measurability with respect to the filtration $(\tilde{\mathcal{F}}_t)$ generated by $\tilde W$.
Quite often, the analysis in \citep{DHK1101,HKM1011} is based on arguments involving supermartingales and their respective right hand limit processes.
However, since in general $(\tilde{\mathcal{F}}_t)$ is neither right- nor left-continuous, we cannot resort to these standard procedures while staying adapted.
Therefore, we adopt the notion of optional strong supermartingales, which, by a result of \citet{Dellacherie1982}, are \ladlag~processes and relieve us of having to take right hand limits.
Accordingly, we formulate our BSDE in a stronger sense, that is with respect to stopping times.
More precisely, a  \ladlag~process $Y$ and a control process $Z$ constitute a supersolution of a backward stochastic differential equation if 
\begin{equation}\label{intro:eq:supersolution}
    Y_\sigma-\int_{\sigma}^{\tau}g_{s}\left(Y_s,Z_s \right)ds+\int_{\sigma}^{\tau}Z_s d\tilde{W}_s\geq Y_{\tau} \quad \text{ and }\quad Y_T\geq \xi,
\end{equation}
for all stopping times $0\leq \sigma\leq \tau\leq T$.
Here, $Y$ and $Z$ are adapted and predictable with respect to $(\tilde{\mathcal{F}}_t)$, respectively, and the equation is to be understood in a $\theta$-wise sense, that is for example $\int Zd\tilde W$ represents the family of projections $(\int Z(\theta)dW(\theta))_{\theta\in\Theta}$.

Our main result proves that under the same conditions on the generator as in \citep{DHK1101}, \citep{HKM1011} or \citep{DKGT2014}, there exists a minimal supersolution to \eqref{intro:eq:supersolution} in the quasi-sure sense among the supersolutions with a regularity controlled in terms of modulus of continuity.
As aforementioned, an appropriate essential infimum is necessary to overcome the lack of a dominating probability measure.
Therefore, we first prove that the pointwise infimum of the regular supersolutions is a good candidate value process that can be approximated by a sequence of supersolutions.
Second, with the candidate value process at hand, we obtain the candidate control process by arguing for each $\theta$ separately and then aggregating similar to \citet{STZ3} and \citet{nutz10} by using a result by \citet{karandikar01}.

The super replication problem under model uncertainty introduced by \citet{lyons95} is relatively recent and has been subject to many studies, see for example, \citet{avellaneda95,denis06,BionKer10,bion10,epstein2011}.
Except for the latter, they all take into account a superhedging problem under volatility uncertainty, whereas the latter also takes into account drift uncertainty.
It happens that the mathematical techniques underlying the problem of superhedging under volatility uncertainty are related to the theory of capacities introduced by \citet{choquet1954} and to quasi-sure stochastic analysis, see \citep{STZ10}, \citet{denis06,peng2011}, and the numerous references therein.
The superhedging problem under volatility uncertainty is also closely linked to other mathematical topics.
On the one hand, to the so called $G$-expectations introduced by \citet{peng07,peng08}, see also \citep{peng2011} and \citet{STZ2} for further studies and references.
On the other hand, to fully non-linear parabolic Partial Differential Equations as introduced by \citet{CSTV2007} and second order Backward Stochastic Differential Equations -- 2BSDE for short -- see \citep{STZ3} for the well posedness, \citet{STZ1} for a dual formulation, and \citet{ST2009} for the corresponding dynamic programing principle.
In contrast to these works, the technique presented here allows to consider generators without growth conditions, and also, no particular stability conditions on the set of volatility models.

The paper is organized as follows. In Section \ref{sec:settingM} we fix our notations and the setting, and introduce our notion of essential infimum. We define minimal supersolutions and introduce our main conditions in Section \ref{sec:tarpo:intro:rob}, which also contains our main result.

\section{Setting and Notation}\label{sec:settingM}
The set of dyadic numbers between $0$ and a finite time horizon $T>0$ is denoted by $\Pi:=\set{kT/n : n\in \N,k=0,\ldots,n}$.
Let $(\Omega,\mathcal{F},P)$ be a probability space carrying a $d$-dimensional Brownian motion $W$.
By $\left( \mathcal{F}_t\right)$ we denote the augmented filtration generated by $W$, which satisfies the usual conditions.
Let $L^0(\mathcal{F}_t)$ denote the set of $\mathcal{F}_t$-measurable random variables, where two of them are identified if they coincide $P$-almost surely.
For $p>0$, the space $L^p({\cal F}_t)$ denotes those random variables in $L^0({\cal F}_t)$ with finite $p$-norm.
We denote by $\mathcal{T}$ the set of $(\mathcal{F}_t)$-stopping times on $\Omega$ with values in $[0,T]$.
An $(\mathcal{F}_t)$-optional process $Y:\Omega\times [0,T]\to \R$ is a \emph{strong supermartingale} if $Y_{\sigma}\in L^1(\mathcal{F}_\sigma)$ and $E\left[ Y_{\tau}\mid \mathcal{F}_{\sigma} \right]\leq Y_{\sigma}$, for all $\sigma,\tau \in \mathcal{T}$ with $\sigma\leq \tau$.

Let $\Theta$ be a family of volatility processes
\begin{equation*}
	\theta:\Omega\times [0,T]\longrightarrow \mathbb{S}^{>0}_d,
\end{equation*}
which are progressively measurable and such that $\int_0^T \|\theta_u^{1/2}\|^2 du<\infty$, $P$-almost surely.
Here, $\mathbb{S}^{>0}_d$ is the set of strictly positive definite $d\times d$-matrices.
On the product space $\tilde{\Omega}:=\Omega \times \Theta$, we consider the process $\tilde{W}:\tilde{\Omega}\times \left[ 0,T \right]\to \R^d$ defined as the stochastic integral
\[\tilde{W}\left( \theta \right)=\int\theta^{1/2} dW,\quad \theta \in \Theta,\]
generating the filtration $(\tilde{\mathcal{F}}_t)$, where $\tilde{\mathcal{F}}_t:=\sigma(\tilde{W}_s ; s\leq t)$.
Since $(\tilde{\cal F}_t)$ is in general not right-continuous, we also consider $(\tilde{\cal F}_t^+)$ defined by
$\tilde {\cal F}_t^+:=\bigcap_{s>t} \tilde{\cal F}_s$, for $t\in [0,T)$, and $\tilde {\cal F}^+_T:=\tilde {\cal F}_T$.

On the sigma-algebra $\tilde{\mathcal F}_T$, it is in general not possible to define a probability measure under which all probability measures $P^\theta[A]:=P[A(\theta)]$, $\theta\in\Theta$, are absolutely continuous where $A(\theta)=\set{\omega \in \Omega: (\omega, \theta)\in A}$.
We therefore define the set function $\tilde{P}:\tilde{\mathcal{F}}_T\to [0,1]$ by 
\begin{equation}
	\tilde{P}\left[A\right]:=\sup_{\theta \in \Theta}P^\theta\left[A\right],\quad A \in \tilde{\mathcal{F}}_T.
	\label{}
\end{equation}
By $C_0([0,T];{\mathbb R}^d)$ we denote the space of continuous functions $w:[0,T]\rightarrow {\mathbb R}^d$, $w(0)=0$, equipped with the uniform norm $||w||_\infty:=\sup_{0\leq t\leq T}|w(t)|$.
\begin{remark}\label{rem:measures}
	For each $\theta\in\Theta$, let $\mu^\theta[B]:=P[\tilde{W}(\theta)\in B]$, where $B\in \mathscr{B}(C_0([0,T];{\mathbb R}^d))$.
	By means of \citep[Theorem 1]{peng2011}, $c(B):=\sup_{\theta\in\Theta} \mu^\theta(B)$ defines a \emph{capacity} 
	on $\mathscr{B}\left(C_0([0,T];{\mathbb R}^d)\right)$. Since any $A\in\tilde{\cal F}_T$ is of the form 
	$A=\tilde W^{-1}(B)$ for some $B \in\mathscr{B}(C_0([0,T];{\mathbb R}^d))$, it follows that 
	$\tilde{P}[A]=c(B)$ is a capacity on $\tilde{\cal F}_T$.
	In applications, the measures $\mu^\theta$, $\theta\in\Theta$, are often mutually singular.
\end{remark}
We work under the following assumption on the measures $\mu^\theta$ defined in Remark \ref{rem:measures}.
\begin{enumerate}[label=\textsc{(RCP)},leftmargin=40pt]
	\item \label{Cond:Comp} the set $\{\mu^\theta:\theta\in\Theta\}$ is relatively weak${}^\ast$-compact\footnote{That is, the $\sigma(\mathcal{M}_1,C_b)$-topology on the probability measures over the Polish space $C_0([0,T];\mathbb{R}^d)$.}.
\end{enumerate}
By means of Prohorov's theorem, the relatively weak${}^\ast$-compactness of $\set{\mu^\theta : \theta \in \Theta}$ is equivalent to the fact that $\set{\mu^\theta:\theta\in\Theta}$ is tight, see also \citep[Theorem 6]{peng2011}.
For instance, \ref{Cond:Comp} is satisfied, if $a \leq \theta \leq b$ for every $\theta \in \Theta$ for constants $0<a\leq b$.

In the following we summarize some notations of capacity theory, see also \cite{peng2011}.
A subset $A$ of $\tilde{\Omega}$ is called a \emph{polar set} if there exists $B \in \tilde{\mathcal{F}}_T$ with $A \subseteq B$ such that $\tilde{P}[B]=0$.
The set of all polar sets is denoted by ${\mathcal N}$.  
We say that a property holds \emph{quasi-surely} if this property holds outside a polar set, that is, this property holds $P^\theta$-almost surely for all $\theta \in \Theta$.
By $L^0(\tilde{\mathcal{F}}_t)$ we denote the set of $\tilde{\mathcal{F}}_t$-measurable random variables $X:\tilde{\Omega}\to \R$, where two of them are identified if they coincide quasi-surely.
Equalities and inequalities between $\tilde{\mathcal{F}}_t$-measurable random variables are understood in the quasi-sure sense.
For any $X \in L^0(\tilde{\mathcal{F}}_T)$ such that $E[X(\theta)]$ exists for all $\theta\in\Theta$, we define the \emph{upper expectation} of $X$ as
\begin{equation}
	\tilde{E}\left[ X \right]:=\sup_{\theta\in\Theta}E\left[ X(\theta) \right].
	\label{}
\end{equation}
The set $L^1(\tilde{\mathcal{F}}_T)$ consists of those $X \in L^0(\tilde{\mathcal{F}}_T)$, for which $\tilde{E}[\abs{X}]< +\infty$.

For any $X\in L^0(\tilde{\mathcal{F}}_t)$, there exists a measurable function $\varphi:C_0([0,T],\R^d)\to \R$ such that $X=\varphi(\tilde{W}^t)$, where $\tilde W^t$ is the stopped process $\tilde W_s^t:=\tilde W_{s\wedge t}$.
In general it is not possible to define an ``essential infimum'' for subsets in $L^0(\tilde{\cal F}_t)$ with respect to the capacity $\tilde{P}$.
However, under the assumption that the infimum is taken over a subset of regular random variables, then it is an essential infimum in the sense of the subsequent proposition.
To this end, we fix an arbitrary countable set $\mathfrak{M}$ of moduli of continuity $\mathfrak{m}$, that is, $\mathfrak{m}: [0,\infty]\to [0,\infty]$ where $\mathfrak{m}(0)=0$ and $\mathfrak{m}$ is continuous at $0$.\footnote{
For instance, $\mathfrak{M}=\set{\mathfrak{m}(x)=qx: q \in \mathbb{Q}_+}$ or  $\mathfrak{M}^\prime=\set{\mathfrak{m}(x)=qx^r : q,r \in \mathbb{Q}_+, 0<r\leq 1}$ correspond to the moduli of continuity of all Lipschitz or H\"older continuous functions, respectively.}
For $\mathfrak{m}\in \mathfrak{M}$, we further define $C_{\mathfrak{m}}(\tilde{\mathcal{F}}_t)$ as the set of those  $X =\varphi(\tilde W^t) \in L^0(\tilde{\mathcal F}_t)$ where $\varphi$ has a modulus of continuity $\mathfrak{m}$, that is $|\varphi(w)-\varphi(w^\prime)|\leq \mathfrak{m}(||w-w^\prime||_{\infty})$, for all $w,w^\prime \in C([0,T]; \mathbb{R}^d)$.
Recall that the infimum of an arbitrary family of functions with moduli of continuity $\mathfrak{m}$ has itself also a modulus of continuity $\mathfrak{m}$ provided it is finite valued at every point.
\begin{proposition}\label{prop:ex:inf}
	Let $\mathcal{X}\subseteq C_{\mathfrak{m}}(\tilde{\mathcal{F}}_t)$, $\mathfrak{m} \in \mathfrak{M}$, bounded from below and define 
    \begin{equation*}
       X^\ast:=\varphi^\ast(\tilde W^t)\in  C_{\mathfrak{m}}(\tilde{\mathcal{F}}_t) 
    \end{equation*}
    where $\varphi^\ast$ is the pointwise infimum over all functions $\varphi:C_0([0,T];{\mathbb R}^d)\to{\mathbb R}$ satisfying $\varphi(\tilde W^t)\in \mathcal{X}$.
    Then, for any $\theta \in \Theta$, there exists a sequence $(X^n)$ in $\mathcal{X}$ such that
	\begin{equation}
		X^\ast(\theta)=\left(\inf_{n\in{\mathbb N}} X^n\right)(\theta).
		\label{eq:blo01}
	\end{equation}
	If in addition \ref{Cond:Comp} is fulfilled, then there exists a sequence $(X^n)$ in ${\cal X}$ such that for every $\varepsilon>0$ it holds
	\begin{equation}\label{eq:blo02}
		\lim_{n\to\infty} \tilde P\left[ (X^1\wedge\cdots\wedge X^n)-X^\ast > \varepsilon\right] = 0.
	\end{equation}
\end{proposition}
\begin{proof}
	\begin{enumerate}[label=\textit{Step \arabic*:},fullwidth]
		\item Fix $\varepsilon>0$ and $\theta\in \Theta$. 
			There exists a compact set $K\in\mathscr{B}(C_0([0,T];{\mathbb R}^d))$ such that $\mu^\theta(K^c)\le\varepsilon$.
			For any $x\in K$ let $\varphi_x^\varepsilon: C_0([0,T];{\mathbb R})\to {\mathbb R}$ be a function with modulus of continuity $\mathfrak{m}$ such that $\varphi^\varepsilon_x(\tilde W^t)\in{\cal X}$ and $|\varphi^\ast(x)-\varphi^\varepsilon_x(x)|\le\varepsilon$, and define the open sets
			\begin{equation*}
				O^\varepsilon_x:=\Set{y\in C_0([0,T];{\mathbb R}^d) : \abs{\varphi^\varepsilon_x(y)-\varphi^\varepsilon_x(x)}<\varepsilon\mbox{ and }\abs{\varphi^\ast(x)-\varphi^\ast(y)}<\varepsilon}.
			\end{equation*}
			The family $(O^\varepsilon_x)_{x\in K}$ is an open cover of $K$, so that by compactness, there exist $x_1,\dots, x_N$ such that	$K\subseteq O^\varepsilon_{x_1}\cup\dots\cup O^\varepsilon_{x_N}$.
			By construction holds $\varphi^\varepsilon_{x_1}\wedge\dots\wedge \varphi^\varepsilon_{x_N}\le \varphi^\ast+3\varepsilon$ on the set $K$.
			Hence
			\[ P\left[\varphi^\varepsilon_{x_1}\left({\tilde W}^t(\theta)\right)\wedge\dots\wedge \varphi^\varepsilon_{x_N}\left({\tilde W}^t(\theta)\right)>\varphi^\ast\left({\tilde W}^t(\theta)\right)+3\varepsilon\right]\le\varepsilon. \]
			This shows that $X^\ast(\theta)=\mathop{\rm ess\,inf}\left\{X\in L^0({\cal F}_t): X\in {\cal X}(\theta)\right\}$ and by \citet[Theorem A.32]{foellmer01} there exists a sequence $(X^n)$ in ${\cal X}$ such that $X^\ast(\theta)=(\inf_n X^n)(\theta)$.

		\item Fix $\varepsilon>0$.
			Since $\set{\mu^\theta:\theta\in\Theta}$ is tight, it follows that there exists a compact set $K\in\mathscr{B}(C_0([0,T];{\mathbb R}^d))$ such that $c(K^c)\le \varepsilon$.
			Let $\varphi^\varepsilon_{x_1},\dots, \varphi^\varepsilon_{x_N}$ be the functions as defined in the previous step, so that 
            \[\tilde P\left[\left(\varphi^\varepsilon_{x_1}({\tilde W}^t )\wedge\dots\wedge \varphi^\varepsilon_{x_N}({\tilde W}^t ) -X^\ast\right) >3\varepsilon\right]\le\varepsilon\]
			Finally, defining  $(X^n)$ as a sequence running through 
			$\bigcup_{n\in\mathbb{N}}\{\varphi^{1/n}_{x_1}({\tilde W}^t ),\dots, \varphi^{1/n}_{x_{N(n)}}({\tilde W}^t )\}$ is as desired.
	\end{enumerate}
\end{proof}


\section{Minimal Supersolutions under Volatility Uncertainty}\label{sec:tarpo:intro:rob}
Let $M, N:\tilde{\Omega}\times [0,T]\to \R$ be $(\tilde{\mathcal{F}}_t)$-adapted processes.
The process $M$ is called \cadlag, \caglad~or \ladlag~if the paths of $M$ are  \cadlag, \caglad~or \ladlag~quasi-surely, respectively.
Given a \ladlag~process, we denote by $M^-$ and $M^+$ its \caglad~and \cadlag~version, respectively, that is
\begin{gather*}
	M^-_t:=\lim_{s\nearrow t}M_s,\quad \text{for }t \in ]0,T],\quad\text{and}\quad M^-_0:=M_0,\\
	M^+_t:=\lim_{s\searrow t}M_s,\quad \text{for }t \in [0,T[,\quad\text{and}\quad M^+_T:=M_T,
\end{gather*}
outside the polar set where $M$ is not \ladlag.
Two $(\tilde{\mathcal{F}}_t)$-adapted processes $M,N:\tilde{\Omega}\times [0,T]\to \R$ are modifications of each others, if $M_t=N_t$, for all $t \in [0,T]$.
We say that $M$ is a \emph{supermartingale} or a \emph{strong supermartingale}, if $M(\theta)$ is a supermartingale or a strong supermartingale, for all $\theta \in \Theta$, respectively.
See \cite[Appendix I]{Dellacherie1982} for a definition of strong supermartingales.

Let us define the following sets of value and control processes:
\begin{itemize}
	\item $\mathcal{S}$ is the set of $(\mathcal{F}_t)$-adapted \ladlag~processes $Y:\Omega \times [0,T]\to \R$;
	\item $\tilde{\mathcal{S}}$ is the set of \ladlag~processes $Y:\tilde{\Omega}\times [0,T]\to \R$, and such that $Y(\theta)$ is optional, for all $\theta \in\Theta$;
	\item For every $\theta \in \Theta$, $\mathcal{L}(\theta)$ is the set of $(\mathcal{F}_t)$-predictable processes $Z:\Omega \times [0,T]\to \R^d$ such that $P\left[\int_{0}^{T}\|Z_u \theta_u^{1/2}\|^2 du<\infty\right]=1$;
	\item $\tilde{\mathcal{L}}$ is the set of $(\tilde{\mathcal{F}}_t)$-predictable processes $Z:\tilde{\Omega}\times [0,T]\to \R^d$ such that $Z(\theta)\in \mathcal{L}(\theta)$, for all $\theta \in \Theta$.
\end{itemize}
A \emph{generator} is a jointly measurable function g from $\tilde \Omega \times [0, T ]\times \mathbb{R}\times \mathbb{R}^{1\times d}$ to $\mathbb{R}\cup  \{+\infty\}$ such that the mapping $(s,\omega, \theta)\mapsto g_s(\omega,\theta,y,z):([0,t]\times \tilde \Omega,\mathcal{B}([0,t])\otimes \tilde{\mathcal{F}}_t)\rightarrow (\mathbb{R}^d,\mathcal{B}(\mathbb{R}^d))$ is measurable, for each $t$, for all $(y,z)\in\mathbb{R}^{d+1}$.
We say that a generator $g$ is
\begin{enumerate}[label=\textsc{(Pos)},leftmargin=40pt]
	\item \label{cond00} positive, if $g\left(\theta, y,z \right)\geq 0$;
\end{enumerate}
\begin{enumerate}[label=\textsc{(Lsc)},leftmargin=40pt]
	\item \label{condlsc} if $(y,z)\mapsto g(\theta,y,z)$ is lower semicontinuous;
\end{enumerate}
\begin{enumerate}[label=\textsc{(Mon)},leftmargin=40pt]
	\item \label{cond03} increasing, if $y\mapsto g\left(\theta,y,z  \right)$ is increasing; 
\end{enumerate}
\begin{enumerate}[label=\textsc{(Mon${}^\prime$)},leftmargin=40pt]
	\item \label{cond02} decreasing, if $y\mapsto g\left(\theta,y,z  \right)$, is decreasing; 
\end{enumerate}
\begin{enumerate}[label=\textsc{(Con)},leftmargin=40pt]
	\item \label{cond01} convex, if $z \mapsto g\left(\theta,y,z\right)$ is convex;
\end{enumerate}
\begin{enumerate}[label=\textsc{(Con${}^\prime$)},leftmargin=40pt]
	\item \label{cond01bis} jointly convex, if $(y,z)\mapsto g\left(\theta,y,z\right)$ is convex;
\end{enumerate}
\begin{enumerate}[label=\textsc{(Nor)},leftmargin=40pt]
	\item \label{cond04} normalized, if $g\left(\theta,y,0 \right)=0$;
\end{enumerate}
$P\otimes dt$-almost surely, for all $y\in\R$, all $z \in \R^{1 \times d}$ and all $\theta \in \Theta$.

A pair $(Y,Z)\in\tilde{{\cal S}}\times \tilde{\mathcal{L}}$ is said to be a \emph{supersolution} of the BSDE with generator $g$ and terminal condition $\xi\in L^0(\tilde{\cal F}_T)$, if
\begin{equation}\label{eq:central:ineq:rob}
	Y_\sigma(\theta)-\int_\sigma^\tau g_{u}(\theta,Y_u(\theta),Z_u(\theta))du+\int_\sigma^\tau Z_{u}(\theta) d\tilde{W}_{u}(\theta) \geq Y_\tau(\theta)\quad\text{and} \quad Y_T(\theta)\geq\xi(\theta),
\end{equation}
for all $\sigma,\tau \in \mathcal{T}$, with $\sigma\leq \tau$, and for all $\theta \in \Theta$.
For such a supersolution $(Y,Z)$, we call $Y$ the \emph{value process} and $Z$ its \emph{control process}.
However, in order to avoid so-called ``doubling strategies'', present even for the simplest generator $g\equiv 0$, see \citet{dudley02} or \citet[Section 6.1]{harrison01}, we only consider control processes, which are admissible, that is $\int Z(\theta)d\tilde W(\theta)$ is a supermartingale, for all $\theta\in\Theta$.
We denote the set of such supersolutions by
\begin{equation}\label{set:rob:A:B}
	\mathcal{A}(\xi)= \set{ (Y,Z) \in\tilde{\mathcal{S}}\times \tilde{\mathcal{L}} : Z \text{ is admissible and } \eqref{eq:central:ineq:rob} \text{ holds}}.
\end{equation}
Our goal is to prove the existence of minimal supersolutions.
In order to make use of the notion of "essential infimum" in the sense of Proposition \ref{prop:ex:inf}, we restrict to the subclass $\mathcal{A}^{\mathfrak{M}}(\xi)$ of those supersolutions $(Y,Z)\in \mathcal{A}(\xi)$ which are $\mathfrak{M}$-regular, that is, $Y$ has a modification $\hat Y$ satisfying $\hat Y_t\in C_{\mathfrak{m}}(\tilde {\cal F}_t)$ for all $t\in [0,T]$ for some $\mathfrak{m}\in\mathfrak{M}$.

The main result of this paper states that the infimum over all $\mathfrak{M}$-regular supersolutions
\begin{equation*}
    \mathcal{E}_t^{\mathfrak{M}}(\xi):=\inf\Set{Y_t: (Y,Z) \in \mathcal{A}^{\mathfrak{M}}(\xi)}
\end{equation*}
is a supersolution, that is, there exists $(Y,Z)\in \mathcal{A}(\xi)$ such that $\mathcal{E}^{\mathfrak{M}}(\xi)$ is a modification of $Y$.
Here, the infimum is understood as the pointwise infimum over the respective representants in $\bigcup_{\mathfrak{m}\in\mathfrak{M}}C_{\mathfrak{m}}(\tilde{\mathcal{F}}_t)$.
The result strongly relies on the following proposition which shows that $\mathcal{E}^{\mathfrak{M}}(\xi)$ can be approximated by a sequence of $\mathfrak{M}$-regular supersolutions.



\begin{proposition}\label{prop:approximation}
 Let $g$ be a generator fulfilling \ref{cond00} and $\xi\in L^0(\tilde{\mathcal{F}}_{T})$ be a terminal condition such that $\xi^- \in L^1(\tilde{\mathcal{F}}_T)$.
Suppose that \ref{Cond:Comp} holds and that there exists $\bar{\theta}\in \Theta$ such that $\mu^{\bar{\theta}}$ is strictly positive, and $\mathcal{A}^{\mathfrak{M}}(\xi)\neq \emptyset$.
Then, there exists a sequence $((Y^n,Z^n))\subseteq \mathcal{A}^\mathfrak{M}(\xi)$ such that $Y:=\inf Y^n\in\tilde {\cal S}$ 
is a modification of $\mathcal{E}^\mathfrak{M}(\xi)$.
\end{proposition}
\begin{proof}

Let  $\mathcal{A}^{\mathfrak{m}}:=\mathcal{A}^{\mathfrak{m}}(\xi)$ be the set of supersolutions $(Y,Z)\in \mathcal{A}(\xi)$ which
are $\mathfrak{m}$-regular, that is, $Y$ has a modification $\hat Y$ satisfying $\hat Y_t\in C_{\mathfrak{m}}(\tilde{\mathcal{F}}_t)$
for all $t\in[0,T]$.

    \begin{enumerate}[label=\textit{Step \arabic*:}, fullwidth,ref=Step \arabic*]
        \item Fix an $\mathfrak{m}\in\mathfrak{M}$ such that ${\cal A}^{\mathfrak{m}}\neq\emptyset$ and define $\mathcal{E}_t^{\mathfrak{m}}=\inf\Set{Y_t: (Y,Z) \in \mathcal{A}^{\mathfrak{m}}}$.
         In this first step, we provide a countable dense subset of paths in $C([0,T];\mathbb{R}^d)$ along which $\mathcal{E}^\mathfrak{m}$ jumps only countably many times.
         By Lemma \ref{lem:rob:Y:satisfies:tarpo}, for any supersolution $(Y,Z)\in{\cal A}^{\mathfrak{m}}$, it holds $Y_t(\theta)\ge -E[\xi^-(\theta)\mid{\cal F}_t]$, for all $\theta\in\Theta$ and $t\in [0,T]$. Hence, by Proposition \ref{prop:ex:inf} there exist 
$\varphi_t:C_0([0,T];{\mathbb R}^d)\to{\mathbb R}$ with modulus of continuity $\mathfrak{m}$ such that  
$\mathcal{E}^{\mathfrak{m}}_t=\varphi_t(\tilde W^t)$, for all $t\in[0,T]$. 
Define the mappings $\varphi^-,\varphi^+:C_0([0,T];\mathbb{R}^d)\rightarrow \mathbb{R}^{[0,T]}$ given by
            \begin{equation}\label{eq:def:phiminusplus}
                w \mapsto \left(\limsup_{\Pi \ni q\uparrow t} \varphi_q(w^q) \right)_{t\in[0,T]} \quad \text{ and } \quad 
                w \mapsto \left(\limsup_{\Pi \ni q\downarrow t} \varphi_q(w^q) \right)_{t\in[0,T]}
            \end{equation}
            where $w^q:=w_{\cdot\wedge q}$, respectively.
            Since $\varphi_t$ has a modulus of continuity $\mathfrak{m}$ for every $t$, it follows that $\varphi^-_t, \varphi^+_t$ also have a modulus of continuity $\mathfrak{m}$ for every $t$.
            For quasi all $w\in C_0([0,T];{\mathbb R}^d)$ the image $\varphi^-(w)$ is \caglad.
            Indeed, note first that, by Lemma \ref{lem:rob:Y:satisfies:tarpo06}, for all $\theta\in\Theta$, $P$-almost surely,
            \begin{equation*}
                \varphi^-\left(\tilde W\left(\theta\right)\right) =  \left(\limsup_{\Pi \ni q\uparrow t} \varphi_q\left(\tilde W^q\left(\theta\right)\right) \right)_{t\in[0,T]}=  \left(\limsup_{\Pi \ni q\uparrow t} \mathcal{ E}^{\mathfrak{m}}_q\left(\theta\right) \right)_{t\in[0,T]}=\mathcal{E}^{\mathfrak{m},-}\left(\theta\right).
            \end{equation*}
            Now, let $N:=\set{w \in C_0([0,T];\mathbb{R}^d): \varphi^-(w) \text{ is not \caglad}}$. 
            Then, again with Lemma \ref{lem:rob:Y:satisfies:tarpo06}, for all $\theta\in\Theta$,
            \begin{equation*}
                P\left[\tilde W\left(\theta\right)\in N\right]= P\left[\varphi^-\left(\tilde W\left(\theta\right)\right) \text{ is not \caglad}\right]=P\left[\mathcal{ E}^{\mathfrak{m},-}\left(\theta\right) \text{ is not \caglad}\right]=0,
            \end{equation*}
            and hence $c(N)=0$.
            By the same arguments we obtain that for quasi all $w\in C_0([0,T];{\mathbb R}^d)$ the image $\varphi^+(w)$ is \cadlag. 
            It follows that for quasi all $w\in C_0([0,T];{\mathbb R}^d)$ the set of jump points
            \begin{equation*}
                \mathcal{J}(w):=\Set{t\in[0,T]:\varphi_t^{-}\left(w\right)> \varphi_t^{+}\left(w\right)}
            \end{equation*}
            is countable.
            Indeed, for $N:=\set{w\in C_0([0,T];{\mathbb R}^d): \mathcal{J}(\omega) \text{ is uncountable}}$ we have
            \begin{equation}\label{eq:proof:rob:min:supsol:1}
                \begin{split}
                    P\left[\tilde W(\theta) \in N\right]&= P\left[\mathcal{J}\left(\tilde W\left(\theta\right)\right) \text{ is uncountable}\right]\\
                    &=P\left[\mathcal{E}^{\mathfrak{m},-}_t\left(\theta\right)> \mathcal{ E}^{\mathfrak{m},+}_t\left(\theta\right) \text{ for uncountably many } t\in [0,T]\right]=0,
                \end{split}
            \end{equation}
            for all $\theta\in\Theta$, which implies $c(N)=0$. 
            To see the last equality in \eqref{eq:proof:rob:min:supsol:1}, note first that, $P$-almost surely, $\mathcal{E}^{\mathfrak{m},-}_t(\theta)(\omega) > \mathcal{ E}^{\mathfrak{m},+}_t(\theta)(\omega)$ implies that $\mathcal{E}^{\mathfrak{m},+}_t(\theta)(\omega)$ jumps at $t$.
            Indeed, suppose that it does not, that is $\mathcal{E}^{\mathfrak{m},+}_t(\theta)(\omega)=\lim_{s\uparrow t}\mathcal{E}^{\mathfrak{m},+}_s(\theta)(\omega)$.
            Then we can find, for every $\varepsilon >0$, some $s\in[0,t)$ and a $p\in\mathbb{Q}$ with $s<p<t$, such that
            \begin{equation*}
                \abs{\mathcal{ E}^{\mathfrak{m},+}_t(\theta)(\omega)-\mathcal{ E}^\mathfrak{m}_p(\theta)(\omega)}\leq \abs{\mathcal{ E}^{\mathfrak{m},+}_t(\theta)(\omega)-\mathcal{ E}^{\mathfrak{m},+}_s(\theta)(\omega)}+\abs{\mathcal{ E}^{\mathfrak{m},+}_s(\theta)(\omega)-\mathcal{E}_p^{\mathfrak{m}}(\theta)(\omega)}\leq \varepsilon.
            \end{equation*}
            Hence, for $\varepsilon_n:=1/n$ and the corresponding $p_n$, with $p_n\leq p_{n+1}$, we obtain the contradiction $\mathcal{ E}^{\mathfrak{m},+}_t(\theta)(\omega)=\lim_n \mathcal{E}^\mathfrak{m}_{p_n}(\theta)(\omega)=\mathcal{ E}^{\mathfrak{m},-}_t(\theta)(\omega)$.
            This implies the result since the \cadlag\,process $\mathcal{ E}^{\mathfrak{m},+}(\theta)$ has only countably many jumps.

            Recall that $C_0([0,T];{\mathbb R}^d)$ is separable and that by assumption there exists $\bar\theta$ such that $\mu^{\bar\theta}$ is strictly positive, that is $\mu^{\bar\theta}(B)>0$, for each nonempty open set $B\in\mathcal{B}(C_0([0,T];\mathbb{R}^d))$.
            This allows us to choose a dense\footnote{That is, the $\|\cdot\|_\infty$-closure of $\{w_k:k\in{\mathbb N}\}$ 
                is $C_0([0,T];{\mathbb R}^d)$.} sequence $(w_k)$ in $C_0([0,T];{\mathbb R}^d)$, such that ${\cal J}(w_k)$ is countable for all $k\in{\mathbb N}$.
            Indeed, we start with an arbitrary dense subset $(\bar w_k)$ and consider the countable set of balls $(B_{1/m}(\bar w_k))_{m,k\in\mathbb{N}}$.
            Each $B_{1/m}(\bar w_k)$ has positive measure under $\mu^{\bar\theta}$ and hence contains some $w_{m,k}$ such that $\mathcal{J}(w_{m,k})$ is countable.
            By construction $(w_{m,k})_{m,k\in\mathbb{N}}$ is a dense subset, which for simplicity is denoted with $(w_{k})$.
            The countable union
            \[ {\cal J}:=\bigcup_{k\in{\mathbb N}} {\cal J}(w_k)\]
            is a countable subset of $[0,T]$.

        \item In this second step, still for a fixed $\mathfrak{m} \in \mathfrak{M}$ with ${\cal A}^{\mathfrak{m}}\neq\emptyset$, we construct an approximating sequence and a limit as in the statement of the proposition but for $\mathcal{E}^\mathfrak{m}$. By \ref{Cond:Comp} and Proposition \ref{prop:ex:inf}, for each $t \in \Pi\cup{\cal J}$ there exists  a sequence $(Y^{n,t},Z^{n,t})_{n\in\N}$ in $\mathcal{A}^\mathfrak{m}$ which satisfies
            \begin{equation*}
                \lim_{n\to\infty}\tilde P\left[( Y^{1,t}_t\wedge\dots\wedge Y^{n,t}_t) - \mathcal{E}^\mathfrak{m}_{t}\ge\varepsilon\right] = 0,\qquad\mbox{for every }\varepsilon>0. 
            \end{equation*}
            Now, let $((Y^{n},Z^{n}))$ be a sequence running through the countable family $((Y^{n,t},Z^{n,t}))_{n\in {\mathbb N},t \in \Pi\cup{\cal J}}$, such that
            \begin{equation}\label{main:eq1}
                \tilde P\left[\left(Y^1_t\wedge\dots\wedge Y^n_t\right)-\mathcal{E}^\mathfrak{m}_t\ge\varepsilon\right]\to 0,\quad\mbox{for all }t\in\Pi\cup{\cal J}\mbox{ and every }\varepsilon>0.	
            \end{equation}
            Defining $Y:=\inf_{n\in{\mathbb N}} Y^n$, it holds $Y_t=\mathcal{E}^\mathfrak{m}_t$, for all $t\in\Pi\cup{\cal J}$.

            We next fix an arbitrary $\theta\in\Theta$ and show that $Y(\theta)$ is a strong supermartingale.
            Indeed, since $Y^{n}(\theta)$ is a strong supermartingale, see Lemma \ref{lem:rob:Y:satisfies:tarpo}, for each $n\in{\mathbb N}$, it follows
            \begin{equation*}
                E\left[ Y_{\tau}(\theta)\mid \mathcal{F}_\sigma \right]\le \inf_{n\in{\mathbb N}} E\left[ Y^{n}_\tau(\theta)\mid \mathcal{F}_\sigma \right]\le \inf_{n\in{\mathbb N}}  Y^n_\sigma(\theta)   \leq Y_\sigma(\theta),
            \end{equation*}
            for all $\sigma,\tau \in \mathcal{T}$ with $\sigma\leq \tau$.
            The integrability condition of $Y$ follows from $Y^1_\tau(\theta)\geq Y_\tau(\theta)$ 
            and the fact that $Y^n_\tau(\theta)$ is uniformly bounded from below by $-E[\xi^-(\theta)\mid \mathcal{F}_\tau]\in L^1(\mathcal{F}_\tau)$,
            for all $\tau \in \mathcal{T}$.

            The process $Y\in\tilde{\cal S}$.
            Indeed, for each $\theta \in\Theta$ the process $Y^{n}(\theta)$ is $(\mathcal{F}_t)$-optional.
            Since $Y$ is the countable infimum over the processes $Y^n$, it follows that $Y(\theta)$ is $(\mathcal{F}_t)$-optional for all $\theta \in \Theta$.
            Thus, we deduce by means of \cite[Appendix 1, Theorem 4, p. 395]{Dellacherie1982} that $Y(\theta)$ is \ladlag , for all $\theta \in \Theta$. 
            This shows that quasi all paths of $Y$ are \ladlag. In particular, since $Y_t=\mathcal{E}^\mathfrak{m}_t$, for all $t\in \Pi$, it follows $Y^-=\mathcal{E}^{\mathfrak{m},-}$ and $Y^+=\mathcal{E}^{\mathfrak{m},+}$. 

            Let us show that $Y_t=\mathcal{E}^\mathfrak{m}_t$, for all $t\in[0,T]$.
            Two distinct cases may happen
            \begin{enumerate}
                \item either $\tilde P[\mathcal{E}^{\mathfrak{m},-}_t > \mathcal{E}^{\mathfrak{m},+}_t]>0$.
                    In this case, recall that $\varphi^-_t,\varphi^+_t$ are continuous, the set 
                    \begin{equation*}
                        \Set{w\in C_0([0,T];{\mathbb R}^d): \varphi_t^{-}(w)> \varphi_t^{+}(w)}.
                    \end{equation*}
                    is open and nonempty.
                    Hence, it contains some $w_{k_0}$ and consequently $t\in{\cal J}$, which implies $Y_t=\mathcal{E}^\mathfrak{m}_t$.
                \item or $\tilde P[\mathcal{E}^{\mathfrak{m},-}_t > \mathcal{E}^{\mathfrak{m},+}_t]=0$, that is $\mathcal{E}^{\mathfrak{m},-}_t = \mathcal{E}^{\mathfrak{m},+}_t$.
                    Since $Y$ is a supermartingale and $(\mathcal{F}_t)$ fulfills the usual conditions, it holds $Y_t^-\ge Y_t$, for all $t\in[0,T]$, see \citet[Proposition 1.3.14]{Karatzas1991}.
                    By Lemma \ref{lem:rob:Y:satisfies:tarpo06}, we get
                    \begin{equation*}
                        \mathcal{E}^{\mathfrak{m},-}_t=Y^-_t\ge Y_t\ge \mathcal{E}^\mathfrak{m}_t\ge \mathcal{E}^{\mathfrak{m},+}_t,
                    \end{equation*}
                    which in turns implies $Y_t=\mathcal{E}^\mathfrak{m}_t$.
            \end{enumerate}

        \item Finally, we construct the approximating sequence for $\mathcal{E}^\mathfrak{M}$. 
            W.l.o.g.~we assume that ${\cal A}^{\mathfrak{m}}\neq \emptyset$ for all $\mathfrak{m}\in\mathfrak{M}$.
            For every $\mathfrak{m}\in \mathfrak{M}$, denote by $((Y^{n,\mathfrak{m}},Z^{n,\mathfrak{m}}))$ the sequence constructed in the previous step so that $Y^\mathfrak{m}=\inf_{n\in \mathbb{N}} Y^{n,\mathfrak{m}}\in\tilde{\cal S}$ is a modification of $\mathcal{E}^\mathfrak{m}$.
            By the same argumentation as in the previous step,
            \begin{equation*}
                Y:=\inf_{\mathfrak{m}\in\mathfrak{M}}Y^{\mathfrak{m}}=  \inf_{n \in \mathbb{N}, \mathfrak{m}\in\mathfrak{M}}Y^{n,\mathfrak{m}} \in \tilde{\mathcal{S}}.
            \end{equation*}
            We are left to show that $Y$ is a modification of $\mathcal{E}^\mathfrak{M}$.
            To this end, for every $t \in [0,T]$,
            \begin{equation*}
                \mathcal{E}_t^\mathfrak{M}(\theta)\leq Y_t(\theta)=\inf_{\mathfrak{m} \in \mathfrak{M}}Y_t^\mathfrak{m}(\theta)=\inf_{\mathfrak{m}\in \mathfrak{M}}\mathcal{E}_t^\mathfrak{m}(\theta)=\mathcal{E}_t^\mathfrak{M}(\theta),\quad P\text{-almost surely for all }\theta \in \Theta,
            \end{equation*}
            showing that $Y$ is a modification of $\mathcal{E}^\mathfrak{M}$.
    \end{enumerate}
\end{proof}

Our main existence result for minimal supersolutions of BSDE under model uncertainty can now be stated as follows. 
\begin{theorem}\label{thm:existence:rob:ladlag}
    Suppose that \ref{Cond:Comp} holds and that there exists $\bar{\theta}\in \Theta$ such that $\mu^{\bar{\theta}}$ is strictly positive.
    Let $g$ be a generator fulfilling  \ref{cond00}, \ref{condlsc}, \ref{cond01} and either \ref{cond03} or \ref{cond02}, and a terminal condition $\xi \in L^0(\tilde{\mathcal{F}}_T)$ such that $\xi^- \in L^1(\tilde{\mathcal{F}}_T)$.
    If $\mathcal{A}^\mathfrak{M}(\xi)\neq \emptyset$, then, there exists a there exists a unique $(Y,Z)\in \mathcal{A}(\xi)$ such that $\mathcal{E}^{\mathfrak{M}}(\xi)$ is a modification of $Y$.
\end{theorem}
\begin{remark}
    The subsequent proof together with the methods and results developped respectively in \citep{HKM1011} and \citep{DKGT2014}, show that the statment of the theorem holds true under either one of the following assumption on the generator:
     \begin{itemize}
        \item $g$ fulfills \ref{cond00}, \ref{condlsc}, and \ref{cond01bis}, see \citep{HKM1011};
        \item $g$ fulfills \ref{cond00}, \ref{condlsc}, and \ref{cond04}, see \citep{DKGT2014}.
    \end{itemize}

    As for the assumption $A^\mathfrak{M}(\xi)\neq \emptyset$, it is fulfilled for a wide class of generators and terminal conditions.
    For instance, if $g$ satisfies \ref{cond04}, then any terminal condition $\xi$ bounded from above by a constant $K$ admits $(Y,Z)=(K,0)$ as supersolution which is of any degree of regularity.
    Indeed, since $g(\theta, y,0)=0$, it follows that 
    \begin{equation*}
        Y_\sigma(\theta) -\int_{\sigma}^{\tau}g(\theta,Y_u(\theta),Z_u(\theta))du +\int_\sigma^{\tau}Z_u(\theta)d\tilde{W}_u(\theta)= K-\int_{\sigma}^{\tau}g(\theta,K,0) du=K=Y_\tau(\theta)
    \end{equation*}
    and $Y_T=K\geq \xi$.
\end{remark}
\begin{proof}
    Set $\mathcal{E}^\mathfrak{M}:=\mathcal{E}^\mathfrak{M}(\xi)$.
    By Lemma \ref{lem:rob:Y:satisfies:tarpo}, for any stopping time $\tau\in{\cal T}$ and any supersolution $(Y,Z)\in{\cal A}(\xi)$ holds $Y_\tau(\theta)\ge -E[\xi^-(\theta)\mid{\cal F}_\tau]$, for all $\theta\in\Theta$.
    In particular, $\mathcal{E}^\mathfrak{M}_t\in L^1(\tilde{\cal F}_t)$, for all $t\in [0,T]$.
Further, by means of Proposition \ref{prop:approximation}, there exists a sequence $((Y^n,Z^n))\subseteq \mathcal{A}^\mathfrak{M}$ such that $Y=\inf_n Y^n \in \tilde{\mathcal{S}}$ and $Y$ is a modification of $\mathcal{E}^\mathfrak{M}$.
    \begin{enumerate}[label=\textit{Step \arabic*:},fullwidth,ref=Step \arabic*]
        \item In this step, we construct for each $\theta \in \Theta$ an admissible control process $Z^\theta\in{\cal L}(\theta)$, such that $(Y(\theta),Z^\theta)$ fulfills \eqref{eq:central:ineq:rob}.
            We start by considering the sequence $(\hat{Y}^n(\theta)):=((Y^n)^+(\theta))$ and the limit $\hat{Y}=\inf_n \hat{Y}^n$.
            Lemma \ref{lem:rob:Y:satisfies:tarpo} implies that $(\hat{Y}^n(\theta),Z^n(\theta))$ fulfills \eqref{eq:central:ineq:rob}, for all $n\in\mathbb{N}$.
            In the following, we argue for a fixed $\theta\in\Theta$, and only indicate dependency on $\theta$ if necessary.

            Given the first set of assumptions on the generator we want to apply the method introduced in \citep{DHK1101} to obtain a process $Z^\theta\in{\cal L}(\theta)$ such that $(\hat{Y}^+(\theta),Z^\theta)$ fulfills \eqref{eq:central:ineq:rob}.
            Therefore, we need to construct a sequence $((\tilde Y^n,\tilde Z^n))\subseteq \mathcal{S}\times \mathcal{L}(\theta)$, such that $\tilde Y^n$ is \cadlag\, and $(\tilde Y^n, \tilde Z^n)$ fulfills \eqref{eq:central:ineq:rob}, for all $n\in\mathbb{N}$, $(\tilde Y^n)$ is monotone decreasing, and $\lim_n \tilde Y^n_t=\hat Y_t(\theta)$, for all $t\in\Pi$.
            We proceed as follows and refer to \citep[Lemma 3.1]{DHK1101} for a justification of the involved pastings. 
            Fix $k\in\mathbb{N}$, $\varepsilon > 0$, and let $\Pi^k:=\{iT/2^k:i=0,\cdots,2^k-1\}$.
            Set $(\tilde Y^{1,0},\tilde Z^{1,0}):=(\hat Y^1(\theta),Z^1(\theta))$ and, for $n\in\mathbb{N}$, $n\geq 2$,
            \begin{align*}
                \tilde Y^{n,0}&:=\tilde Y^{n-1,0}1_{[0,\tau_0^n[}+\hat Y^n(\theta)1_{[\tau_0^n,T]},\\
                \tilde Z^{n,0}&:=\tilde Z^{n-1,0}1_{[0,\tau_0^n]}+Z^n(\theta)1_{]\tau_0^n,T]},
            \end{align*}
            where $\tau_0^n:=\inf\{t\geq 0 : \tilde Y^{n-1,0}_t > \hat Y_t^n(\theta)\}$.
            By construction holds $\lim_n \tilde Y^{n,0}_0=\hat Y_0(\theta)$ and we may choose $n_0\in\mathbb{N}$ such that $\tilde Y^{n_0,0}_0-\varepsilon \leq \hat Y_0(\theta)$.
            Set $(\tilde Y^{\varepsilon,0},\tilde Z^{\varepsilon,0}):=(\tilde Y^{n_0,0},\tilde Z^{n_0,0})$.
            Now, let $(\tilde Y^{0,1},\tilde Z^{0,1}):=(\tilde Y^{\varepsilon,0},\tilde Z^{\varepsilon,0})$ and set, for $n\in\mathbb{N}$, $n\geq 1$,
            \begin{align*}
                \tilde Y^{n,1}&:=\tilde Y^{n-1,1}1_{[0,\tau_1^n[}+\hat Y^n(\theta)1_{[\tau_1^n,T]},\\
                \tilde Z^{n,1}&:=\tilde Z^{n-1,1}1_{[0,\tau_1^n]}+ Z^n(\theta)1_{]\tau_1^n,T]},
            \end{align*}
            where $\tau_1^n:=\inf\{t\geq 1T/2^k : \tilde Y^{n-1,1}_t > \hat Y_t^n(\theta)\}$.
            By construction holds $\lim_n \tilde Y^{n,1}_{T/2^k}=\hat Y_{T/2^k}(\theta)$ and using the same arguments as in \citep[Proposition 3.2.2]{DHK1101} we may then construct $(\tilde Y^{\varepsilon,T/2^k},\tilde Z^{\varepsilon,T/2^k})$ such that $Y^{\varepsilon,T/2^k}_{iT/2^k}-\varepsilon \leq \hat Y_{iT/2^k}(\theta)$, for $i=0,1$.
            The continuation of this procedure yields a pair $(\tilde Y^{\varepsilon,\Pi^k},\tilde Z^{\varepsilon,\Pi^k})$ such that $Y^{\varepsilon,\Pi^k}_{t}-\varepsilon \leq \hat Y_{t}(\theta)$, for all $t\in\Pi^k$.
            Let now $((\tilde Y^n,\tilde Z^n):=(\tilde Y^{1/n,\Pi^n},\tilde Z^{1/n,\Pi^n}))$.
            Then, $((\tilde Y^n,\tilde Z^n))$ fulfills all the requirements, except that it needs not be monotone decreasing.
            However, this can be achieved by the same pasting arguments as in the last part of Step 2 in the proof of \citep[Theorem 4.1]{DHK1101}.
            We denote the resulting sequence again with $((\tilde Y^n,\tilde Z^n))$ and observe that the method in \citep[Theorem 4.1]{DHK1101} yields $Z^\theta\in\mathcal{L}(\theta)$ such that $(\tilde Y^+, Z^\theta)$ fulfills \eqref{eq:central:ineq:rob}, where $\tilde Y^+$ is the right hand limit process of the monotone limit $\tilde Y=\lim_n \tilde Y^n$.
            Consequently, since $\tilde Y$ coincides with $\hat Y(\theta)$ on all dyadic rationals, we obtain that $(\hat Y^+(\theta), Z^\theta)$ fulfills \eqref{eq:central:ineq:rob}.

            Now, we show that $\hat{Y}^+_t(\theta)=Y^+_t(\theta)$, for all $t \in [0,T]$ and $\theta \in \Theta$.
            On the one hand, from $Y^n_t(\theta)\geq \hat{Y}_t^n(\theta)$, see Lemma \ref{lem:rob:Y:satisfies:tarpo}, follows $Y_t(\theta)\geq \hat{Y}_t(\theta)$ and $Y_t^+(\theta)\geq \hat{Y}^+_t(\theta)$.
            On the other hand, \eqref{eq:central:ineq:rob} implies, for all $s\geq t$, and $\theta \in\Theta$,
            \begin{equation*}
                \hat{Y}_t^n(\theta)\geq -\int_t^sZ^n_u(\theta)d\tilde{W}_u(\theta)+Y_s^n(\theta).
            \end{equation*}
            By taking conditional expectation we obtain $\hat{Y}^n_t(\theta)\geq E[Y_s^n(\theta)\mid \mathcal{F}_t]$.
            This yields 
            \begin{equation*}
                \hat{Y}_t(\theta)\geq \inf_nE[Y_s^n(\theta)\mid \mathcal{F}_t]\geq E[Y_s(\theta)\mid \mathcal{F}_t].
            \end{equation*}
            Since $Y(\theta)\geq E[\xi(\theta)\mid\mathcal{F}_{\cdot}]$ we may apply Fatou's lemma and obtain, by sending $s$ to $t$, that $\hat{Y}_t(\theta)\geq Y_t^+(\theta)$, which in turn implies $\hat{Y}_t^+(\theta)\geq Y_t^+(\theta)$, for all $t \in [0,T]$.
            Hence $Y_t^+(\theta)=\hat{Y}_t^+(\theta)$, for all $t \in [0,T]$, and we deduce that $(Y^+(\theta),Z^\theta)$ fulfills \eqref{eq:central:ineq:rob}.

            It remains to show that $(Y(\theta),Z^\theta)$ fulfills \eqref{eq:central:ineq:rob}, for all $\theta\in\Theta$.
            To that end note that, since $Y(\theta)$ is a strong supermartingale and $(\mathcal{F}_t)$ fulfills the usual conditions, by \cite[Appendix 1, Remark 5.c, p. 397]{Dellacherie1982} it holds $Y^-_\tau(\theta)\geq Y_\tau(\theta)\geq Y_\tau^+(\theta)$, for all stopping times $\tau \in\mathcal{T}$, and by similar arguments as in Lemma \ref{lem:rob:Y:satisfies:tarpo} we have $Y(\theta)=Y^+(\theta)$, $P\otimes dt$-almost surely.
            Since every $(\mathcal{F}_t)$-stopping time is predictable we may choose, for $0\leq \sigma\leq \tau\leq T$ with $\sigma,\tau \in \mathcal{T}$, an increasing sequence $(\tau_n)$ of stopping times converging to $\tau$, with $\tau_n<\tau$, for all $n\in\mathbb{N}$.
            This yields, for all $\theta \in \Theta$,
            \begin{multline*}
                Y_\sigma(\theta)-\int_\sigma^\tau g_{u}(\theta,Y_u(\theta),Z_u^\theta)du+\int_\sigma^\tau Z_{u}^\theta d\tilde{W}_{u}(\theta)\\
                \geq \lim_{n}Y_\sigma^+(\theta)-\int_\sigma^{\tau_n} g_{u}(\theta, Y_u^+(\theta),Z_u^\theta)du+\int_\sigma^{\tau_n} Z_{u}^\theta d\tilde{W}_{u}(\theta)\\
                \geq \lim_{n}Y_{\tau_n}^+(\theta)=(Y^+_{\tau})^-(\theta)=Y_{\tau}^-(\theta)\geq Y_{\tau}(\theta),
            \end{multline*}
            where the second equality follows from the \ladlag\,property of $Y$. 
            Thus, $(Y(\theta),Z^\theta)$ fulfills \eqref{eq:central:ineq:rob}, for all $\theta\in\Theta$.

        \item In this second and final step, we provide $Z \in \tilde{\mathcal{L}}$ such that $Z(\theta)=Z^\theta$, for all $\theta \in \Theta$.
            The argumentation of  this aggregation result relies on a result in\cite{karandikar01} extended in the present context in \cite{STZ3} and \cite{nutz10}.
            Since $Y^+$ is \cadlag\,and $(Y^+(\theta),Z^\theta)$ fulfills \eqref{eq:central:ineq:rob}, 
            we know that $\langle Y^+(\theta),\tilde{W}(\theta)\rangle=\int_{}^{}Z^\theta \theta du$ and that
            \begin{equation}
                \langle Y^+(\theta),\tilde{W}(\theta)\rangle=Y^+(\theta)\tilde{W}(\theta)-\int_{}^{}Y^-(\theta)d\tilde{W}(\theta)
                -\int_{}^{}\tilde{W}(\theta)dY^+(\theta),\quad \mbox{for all }\theta \in \Theta.
                \label{}
            \end{equation}
            We next argue that the right hand side of the previous expression is $(\tilde{\cal F}_t^+)$-adapted. Indeed, 
            the process $Y^+\tilde{W}$ is $(\tilde{\mathcal{F}}_t^+)$-adapted and 
            since $Y^-$ and $\tilde{W}$ are \caglad, we know by \citep{karandikar01} that there exists an 
            $(\tilde{\mathcal{F}}_t^+)$-adapted process $I$ which coincides with the integral terms $\theta$-wise in the $P$-almost 
            sure sense. We briefly expose how one constructs such a functional for the first integral term.
            For each $n \in \N$, we consider the sequence of $(\tilde{\mathcal{F}}_t^+)$-stopping times $\tilde{\tau}^n_0=0$ and $\tilde{\tau}^n_{k+1}=\inf\{t\geq \tilde{\tau}^n_k:\abs{Y^+_t-Y_{\tilde{\tau}^n_k}^+}\geq 2^{-n}\}$.
            We then define the process $I^n$ through
            \begin{equation}
                I^n_t:=Y^+_{\tilde{\tau}^n_k}+\sum_{i=0}^{k-1}Y_{\tilde{\tau}^n_i}^+\left( \tilde{W}_{\tilde{\tau}^n_{i+1}}-\tilde{W}_{\tilde{\tau}^n_i} \right),\quad \text{ for }\tilde{\tau}_k^n\leq t< \tilde{\tau}^n_{k+1},\text{ and }k\geq 0.
                \label{}
            \end{equation}
            By construction, $I^n$ is an $(\tilde{\mathcal{F}}_t^+)$-adapted process and we define $I=\limsup_n I^n$ which is also 
            $(\tilde{\mathcal{F}}^+_t)$-adapted.
            By use of the Burkholder-Davis-Gundy inequality\footnote{For any $(\mathcal{F}_t)$-adapted process $X^\theta$ holds 
                $E\left[ \sup_{t \in [0,T]}\abs{\int_{0}^{t}X^\theta_u d\tilde{W}_u(\theta)} \right]
                \leq C E\left[\left( \int_{0}^{T}\abs{X_u^\theta}^2\theta du\right)^{1/2} \right]$.} holds
            \begin{equation}
                E\left[ \sup_{t \in [0,T]} \abs{I^n_t(\theta)-\int_{0}^{t}Y_u^-(\theta)d\tilde{W}_u(\theta)}\right]\leq 
                C 2^{-n}E\left[\left(\int_{0}^{T}\theta_u du\right)^{1/2}\right].
                \label{}
            \end{equation}
            Since the right hand side of the previous inequality converges to $0$ for each $\theta \in \Theta$, it follows that
            $I$ is an $(\tilde{\mathcal{F}}_t^+)$-adapted process 
		such that $I(\theta)=\int_{}^{}Y^-(\theta)d\tilde{W}(\theta)$, for all $\theta \in \Theta$.

		Hence, there exists an $(\tilde{\mathcal{F}}_t^+)$-adapted ${\mathbb R}^d$-valued process 
		denoted by $\langle Y^+,\tilde{W}\rangle$, which $\theta$-wise coincides with $\int_{}^{}Z^\theta \theta du$.
		Since $\langle Y^+,\tilde{W}\rangle$ is $\theta$-wise continuous, we deduce that it is $(\tilde{\mathcal{F}}_t^+)$-predictable, which implies, see \cite[IV.61 Remark (c)]{Dellacherie1978}, that it is $(\tilde{\mathcal{F}}_t)$-predictable.
		The same argumentation holds for $\langle \tilde{W},\tilde{W}\rangle$, for which holds $\langle \tilde{W}(\theta),\tilde{W}(\theta)\rangle =\int_{}^{}\theta du$.
		We define $Z$ by the pathwise left derivatives, which by means of Lebegue's derivative theorem exists $dt$-almost surely, as follows
		\begin{equation}
			Z_t:=\left( \lim_{h\searrow 0}\frac{\langle Y^+, \tilde{W} \rangle_{t-h}
					-\langle Y^+, \tilde{W} \rangle_{t}}{h} \right)\left( \lim_{h\searrow 0}\frac{\langle \tilde{W},\tilde{W}\rangle_{t-h} -\langle \tilde{W},\tilde{W}\rangle_t}{h}\right)^{-1},\quad t \in ]0,T],
		\end{equation}
		and so $Z$ is $(\tilde{\mathcal{F}}_t)$-predictable.
		Thus, we obtain some $Z \in \tilde{\mathcal{L}}$ such that $Z(\theta)=Z^\theta$ for all $\theta \in \Theta$. 
	\item From the previous argumentation we know that $(Y^+,Z)$ fulfills \eqref{eq:central:ineq:rob}.
		Hence, uniqueness of $Z$ follows from the Doob-Meyer decomposition under each $\theta\in\Theta$, see \citep[Lemma 3.3]{DHK1101} for details.
\end{enumerate}
\end{proof}

\begin{appendix}
\section{Auxiliary Results}
In the following, we state two technical results that are $\theta$-wise argumentations similar to \citep{DHK1101}.
\begin{lemma}\label{lem:rob:Y:satisfies:tarpo}
    Let $g$ be a generator fulfilling \ref{cond00}, and $\xi\in L^0(\tilde{\mathcal{F}}_{T})$ be a terminal condition such that $\xi^-(\theta) \in L^1(\mathcal{F}_T)$, for all $\theta \in\Theta$.
    Let $(Y,Z) \in \mathcal{A}(\xi)$.
    Then $\xi(\theta) \in L^1(\mathcal{F}_T)$, for all $\theta \in\Theta$, and
    \begin{enumerate}[label=(\roman*)]
        \item\label{lem:rob:Y:satisfies:tarpo:01} the value process $Y$ is a strong supermartingale such that $Y_\sigma(\theta)\geq -E[\xi^-(\theta)\mid \mathcal{F}_\sigma]$, for all $\sigma \in \mathcal{T}$, and all $\theta \in \Theta$.
        \item\label{lem:rob:Y:satisfies:tarpo:04} it holds $Y^-_\sigma(\theta)\geq Y_\sigma(\theta)\geq Y^+_\sigma(\theta)$, for all $\sigma \in \mathcal{T}$, and all $\theta \in\Theta$.
            Moreover, we have $Y(\theta)=Y^+(\theta)$, $P\otimes dt$-almost surely, and $(Y^+(\theta),Z(\theta))$ fulfills \eqref{eq:central:ineq:rob}.
    \end{enumerate}
\end{lemma}
\begin{proof}
    As for Item \ref{lem:rob:Y:satisfies:tarpo:01}, from \eqref{eq:central:ineq:rob} and the positivity of the generator follows
    \begin{equation}
        Y_{0}(\theta)+\int_{0}^{\tau}Z_u(\theta)d\tilde{W}_u(\theta)\geq  Y_{\tau}(\theta)\geq- \xi^-(\theta)- \int_{\tau}^T Z_u(\theta)d\tilde{W}_u(\theta),
        \label{eq:lem:001a}
    \end{equation}
    for all $\tau \in \mathcal{T}$ and $\theta \in \Theta$.
    Both sides being integrable by assumption, so is $Y_{\tau}(\theta)\in L^1(\mathcal{F}_\tau)$.
    Since $\xi^-(\theta)\leq \xi(\theta) \leq Y_T(\theta)$, we deduce $\xi(\theta) \in L^1(\mathcal{F}_T)$, for all $\theta \in\Theta$.
    Furthermore, from the admissibility of $Z$ follows $Y_{\tau}(\theta)\geq -E[\xi^-(\theta)\mid \mathcal{F}_{\tau} ]$.
    Similar to \eqref{eq:lem:001a} we deduce that
    \begin{equation}
        Y_{\sigma}(\theta)\geq Y_{\tau}(\theta)-\int_{\sigma}^{\tau}Z_{u}(\theta)d\tilde{W}_u(\theta),
        \label{eq:lem:01}
    \end{equation}
    for all stopping times $0\leq \sigma\leq \tau \leq T$.

    As for Item \ref{lem:rob:Y:satisfies:tarpo:04}, the first statement follows from \cite[Appendix 1, Remark 5.c, p. 397]{Dellacherie1982},  since $Y(\theta)$ is a strong supermartingale and $(\mathcal{F}_t)$ fulfills the usual conditions. 
    To see the second statement, note that by the \ladlag\, property of $Y$ holds $(Y_\sigma^+(\theta))^-=Y_\sigma^{-}(\theta)$, for all $\sigma\in\mathcal{T}$.
    Consequently, for $\sigma\in\mathcal{T}$ such that $(Y_{\sigma}^+(\theta))^-=Y^+_{\sigma}(\theta)$, that is $Y^+(\theta)$ does not jump at $\sigma$, we have $Y_\sigma^-(\theta)=Y_\sigma(\theta)=Y_\sigma^+(\theta)$, that is $Y(\theta)$ does not jump at $\sigma$.
    Denote with $(\tau^n)$ the sequence of stopping times which exausts the jumps of $Y^+(\theta)$, see \citet[Theorem IV.88B]{Dellacherie1978}.
    Then, the process $\bar Y^\theta$ defined by $\bar Y^\theta_t:= Y^+_t(\theta)+\sum_n1_{[\tau^n]}(t)(Y_{\tau^n}(\theta)-Y_{\tau^n}^+(\theta))$, for all $t\in[0,T]$, is an optional modification of $Y(\theta)$.
    Moreover, it holds $\bar Y_\sigma^\theta=Y_\sigma(\theta)$, for all $\sigma\in\mathcal{T}$.
    Hence, by \citep[Theorem IV.86]{Dellacherie1978} $\bar Y^\theta$ is indistinguishable from $Y(\theta)$.
    Since, by definition $\bar Y^\theta=Y^+(\theta)$, $P\otimes dt$-almost surely, we conclude $Y(\theta)=Y^+(\theta)$, $P\otimes dt$-almost surely.
    Finally, for any $\sigma,\tau\in\mathcal{T}$ let $(\sigma_k)$ be a sequence of stopping times decreasing to $\sigma$.
    Then,
    \begin{multline*}
        Y_\sigma^+(\theta)-\int_\sigma^\tau g_{u}(\theta,Y^+_u(\theta),Z_u(\theta))du+\int_\sigma^\tau Z_{u}(\theta) d\tilde{W}_{u}(\theta)\\
        = \lim_{k}Y_{\sigma_k}(\theta)-\int_{\sigma_k}^{\tau} g_{u}(\theta, Y_u(\theta),Z_u(\theta))du+\int_{\sigma_k}^{\tau} Z_{u}(\theta) d\tilde{W}_{u}(\theta)
        \geq \lim_{k}Y_{\tau}(\theta)\geq Y^+_{\tau}.
    \end{multline*}
\end{proof}

Let  $\mathcal{A}^{\mathfrak{m}}$ be the set of supersolutions $(Y,Z)\in \mathcal{A}(\xi)$ which
are $\mathfrak{m}$-regular, that is, $Y$ has a modification $\hat Y$ satisfying $\hat Y_t\in C_{\mathfrak{m}}(\tilde{\mathcal{F}}_t)$ for all $t\in[0,T]$. Define $\mathcal{E}_t^{\mathfrak{m}}=\inf\Set{Y_t: (Y,Z) \in \mathcal{A}^{\mathfrak{m}}}$.

\begin{lemma}\label{lem:rob:Y:satisfies:tarpo06}
    Let $g$ be a generator fulfilling \ref{cond00}, and $\xi\in L^0(\tilde{\mathcal{F}}_{T})$ be a terminal condition such that $\xi^-(\theta) \in L^1$, for all $\theta \in\Theta$.
    Suppose that $\mathcal{A}^{\mathfrak{m}}\neq \emptyset$.
    Then $\mathcal{E}^{\mathfrak{m}}$ is a supermartingale, and the limits
    \begin{equation}\label{def:e-+}
            \mathcal{E}^{\mathfrak{m},-}_t:=\lim_{s\uparrow t,s\in \Pi}\mathcal{E}^{\mathfrak{m}}_s,\quad \text{and}\quad \mathcal{E}_t^{\mathfrak{m,+}}:=\lim_{s\downarrow t,s\in \Pi}\mathcal{E}^{\mathfrak{m}}_s
    \end{equation}
    exist, for all $t\in]0,T[$, quasi-surely. 
    
    Moreover, $\mathcal{E}^{\mathfrak{m},-}$ and $\mathcal{E}^{\mathfrak{m},+}$ are \caglad\, and \cadlag\,supermartingales respectively,\footnote{With the convention that $\mathcal{E}^{\mathfrak{m},-}_0:=\mathcal{E}^{\mathfrak{m}}_0$, and $\mathcal{E}^{\mathfrak{m},+}_T:=\mathcal{E}^{\mathfrak{m}}_T$.} which satisfy
    \begin{equation}\label{E-geE+}
        \mathcal{E}^{\mathfrak{m},-} \geq \mathcal{E}^{\mathfrak{m},+}\quad\mbox{and}\quad
        \mathcal{E}^{\mathfrak{m},-}_t\geq \mathcal{E}^{\mathfrak{m}}_t\geq \mathcal{E}^{\mathfrak{m},+}_t,\quad\mbox{for all }t\in[0,T].
    \end{equation}
\end{lemma}
\begin{proof}
    Note first that $\mathcal{E}^{\mathfrak{m}}$ is adapted by definition.
    Furthermore, given $(Y,Z) \in \mathcal{A}^\mathfrak{m}\neq \emptyset$, Lemma \ref{lem:rob:Y:satisfies:tarpo} implies $\xi(\theta)\in L^1(\mathcal{F}_T)$ and $Y_t(\theta)\geq -E\left[ \xi^-(\theta)\mid\mathcal{F}_t \right]$, for all $\theta\in\Theta$.
    Hence $Y_{t}(\theta)\geq\mathcal{E}^\mathfrak{m}_{t}(\theta)\geq -E[\xi^-(\theta)\mid\mathcal{F}_{t}]$ and $\mathcal{E}^\mathfrak{m}_{t}(\theta)\in L^1(\mathcal{F}_T)$, for all $\theta\in\Theta$.\\
    Fix $\theta \in \Theta$.
    We show that given $t \in [0,T]$ and $\varepsilon>0$ there exists $(Y^\varepsilon,Z^\varepsilon)\in \mathcal{S}\times \mathcal{L}(\theta)$ fulfilling \eqref{eq:central:ineq:rob}, $Y^\varepsilon_t\leq \mathcal{E}^\mathfrak{m}_t(\theta)+\varepsilon$ and $Y^\varepsilon_s\geq \mathcal{E}^\mathfrak{m}_s(\theta)$, for all $s \in [0,T]$.
    By means of Proposition \ref{prop:ex:inf}, there exists a sequence $(Y^n,Z^n)\in \mathcal{A}^\mathfrak{m}$ such that $\mathcal{E}^{\mathfrak{m}}_t(\theta)=(\inf_n Y^n_t)(\theta)$ and $\mathcal{E}^\mathfrak{m}_s\leq Y^n_s$, for all $s \in [0,T]$.
    From this sequence, we define recursively $(\tilde{Y}^n,\tilde{Z}^n)\in \mathcal{S}\times \mathcal{L}(\theta)$ starting with $\tilde{Y}^0=Y^0(\theta)$ and $\tilde{Z}^0=Z^0(\theta)$ and
    \begin{align*}
        \tilde{Y}^n&=Y^0(\theta)1_{[0,t[}+\tilde{Y}^{n-1}1_{\set{\tilde{Y}_t^{n-1}<Y^n_t(\theta)}}1_{[t,T]}+Y^{n}(\theta)1_{\set{\tilde{Y}_t^{n-1}\geq Y^n_t(\theta)}}1_{[t,T]},\\
        \tilde{Z}^n&=Z^0(\theta)1_{[0,t]}+\tilde{Z}^{n-1}1_{\set{\tilde{Y}_t^{n-1}<Y^n_t(\theta)}}1_{]t,T]}+Z^{n}(\theta)1_{\set{\tilde{Y}_t^{n-1}\geq Y^n_t(\theta)}}1_{]t,T]},
    \end{align*}
    for $n\geq 1$.
    It is clear that $(\tilde{Y}^n,\tilde{Z}^n)\subseteq \mathcal{S}\times \mathcal{L}(\theta)$ and fulfills \eqref{eq:central:ineq:rob}.
    By construction, $(\tilde{Y}^n_t)$ is decreasing and such that $\mathcal{E}^{\mathfrak{m}}_t(\theta)=\inf_n \tilde{Y}^n_t$ and $\tilde{Y}^n_s\geq \mathcal{E}^{\mathfrak{m}}_s(\theta)$, for all $s \in [0,T]$.
    Moreover, \citep[Lemma 3.1]{DHK1101} shows that $(Y^\varepsilon,Z^\varepsilon)$ defined as
    \begin{align*}
        Y^\varepsilon&=\tilde{Y}^01_{[0,t[}+\sum_{n}\tilde{Y}^n1_{[t,T]}1_{B^n},\\
        Z^\varepsilon&=\tilde{Z}^01_{[0,t]}+\sum_{n}\tilde{Z}^n1_{]t,T]}1_{B^n},
    \end{align*}
    where $B^0=A^0$, $B^n=A^n\setminus A^{n-1}$, and $A^n=\set{Y^n_t\leq \mathcal{E}^{\mathfrak{m}}_t(\theta)+\varepsilon}$, for $n\in \N$, is such that $(Y^\varepsilon,Z^\varepsilon)\in \mathcal{S}\times \mathcal{L}(\theta)$, fulfills \eqref{eq:central:ineq:rob} and by construction fulfills $Y^\varepsilon_t\leq \mathcal{E}^{\mathfrak{m}}_t(\theta)+\varepsilon$. 

    For $\varepsilon>0$, and any $0\leq s<t\leq T$ we pick $(Y^\varepsilon,Z^\varepsilon)\in \mathcal{S}\times \mathcal{L}(\theta)$ fulfilling \eqref{eq:central:ineq:rob} such that $Y^\varepsilon_s\leq \mathcal{E}^\mathfrak{m}_s(\theta)+\varepsilon$ and $Y^\varepsilon_t\geq \mathcal{E}^\mathfrak{m}_t(\theta)$, for all $t \in [0,T]$.
    Hence
    \begin{equation}\label{eq:tresimportant}
        \mathcal{E}^\mathfrak{m}_t(\theta)\leq Y^\varepsilon_t\leq Y^\varepsilon_s-\int_{s}^{t}g_u(Y_u^\varepsilon,Z^\varepsilon_u)du+\int_{s}^{t}Z^\varepsilon_u d\tilde{W}_u(\theta)      \leq \mathcal{E}^\mathfrak{m}_s(\theta)+\int_{s}^{t}Z^\varepsilon_u d\tilde{W}_u(\theta)+\varepsilon.
    \end{equation}
    Taking conditional expectation on both sides under $\mathcal{F}_s$ followed by sending $\varepsilon$ to zero shows the supermartingale property for $\mathcal{E}^{\mathfrak{m}}(\theta)$.
    Hence, $\mathcal{E}^{\mathfrak{m}}$ is a supermartingale and the definition of $\tilde P$ immediately yields that $\tilde P[A]=0$, where $A\in\mathcal{\tilde F}_T$ is the set where the limits in \eqref{def:e-+} do not exist.
\end{proof}

\end{appendix}

\bibliographystyle{abbrvnat}
\bibliography{bibliography}
\end{document}